\documentclass[11pt]{article}
\usepackage{amsthm,geometry,amssymb,amsmath,enumerate,float,tikz, cite,fp,setspace,comment, calc}
\tikzstyle{every node}=[circle, draw, inner sep=0pt, minimum width=4pt]
\RequirePackage[utf8]{inputenc}

\usepackage{pifont}
\newcommand{\cmark}{\ding{51}}%
\newcommand{\xmark}{\ding{55}}%

\newtheorem{theorem}{Theorem}

\newtheorem{lemma}{Lemma}
\newtheorem{claim}{Claim}

\title{Uniquely restricted matchings in subcubic graphs \\ without short cycles}
\author{M. F\"{u}rst  \and D. Rautenbach}
\date{}

\begin{document}
\maketitle
\begin{center}
{\small 
Institute of Optimization and Operations Research, Ulm University, Germany\\
\texttt{maximilian.fuerst,dieter.rautenbach@uni-ulm.de}\\[3mm]
}
\end{center}

\begin{abstract}
A matching $M$ in a graph $G$ is uniquely restricted if no other matching in $G$
covers the same set of vertices. We prove that any connected subcubic graph
with $n$ vertices and girth at least $5$ contains a uniquely restricted matching of size at least
$(n-1) / 3$ except for two exceptional cubic graphs of order $14$ and $20$.
\end{abstract}
{\small 
\begin{tabular}{lp{13cm}}
\end{tabular}
} \\
{\small 
\begin{tabular}{lp{13cm}}
\end{tabular}
}
\clearpage
\section{Introduction}
We consider simple, finite, and undirected graphs, and 
use standard terminology.
A matching $M$ in a graph $G$ is \textit{uniquely restricted} \cite{ghl} if 
no other matching in $G$ covers the same set of vertices.
The maximum cardinalities of a matching and a uniquely
restricted matching in $G$ will be denoted by 
$\nu(G)$ and $\nu_{ur}(G)$,
respectively.
While determining the matching number is tractable \cite{lp},
determining the uniquely restricted matching number is NP-hard \cite{ghl}.
In the present paper, we establish a tight lower bound on the uniquely restricted
matching number in subcubic graphs of girth at least $5$.
The search for bounds in graphs of bounded degree and large girth is a common task 
for various hard graph invariants, cf. \cite{cdr, fh, ht, hr, lprr, s}.
In \cite{fr} we proved that 
\begin{align} \label{motivation}
 \nu_{ur}(G) \geq \frac{n(G)-1}{\Delta}
\end{align}
for a connected graph $G$
of order $n(G)$,
maximum degree at most $\Delta$ for some $\Delta\geq 4$,
and girth at least $5$,
and we conjectured that (\ref{motivation}) also holds for subcubic graphs,
which was recently confirmed for graphs of girth at least $7$ \cite{fhr}.
In the present paper, we verify our conjecture except for two small graphs     
illustrated in Figure \ref{fig_excub}.
\begin{figure}[H] 
 \begin{minipage}[b]{0.49\textwidth}
 \centering\tiny
 \begin{tikzpicture}[scale = 1.2]
	    \node (u1) at (0,0) {};
	    \node (u2) at (2,0) {};
	    \node (u3) at (1,2) {};
	    \node (u4) at (1,0.75) {};
	    \node (v1) at (1,0) {};
	    \node (v2) at (1.5,1) {};
	    \node (v3) at (0.5,1) {};
	    \node (v4) at (0.5,0.375) {};
	    \node (v5) at (1.5,0.375) {};
	    \node (v6) at (1,1.375) {};
	    
	    \node (w1) at (0,2) {};
	    \node  (w2) at (2,2) {};
	    \node  (w3) at (2,1.375/2) {};
	    
	    \node (x) at (2.5,2.5) {};

	    \foreach \from/\to in {u1/v1, u2/v2, u3/v3, v1/u2, v2/u3, v3/u1, v4/u4, v5/u4, v6/u4, w1/v4, w2/v5, x/w2, x/w3}
	    \draw [-] (\from) -- (\to);

	    \foreach \from/\to in {u1/v4, u2/v5, u3/v6, w1/x}
	    \draw [-, dashed] (\from) -- (\to);
	    
	    \draw[-] (w2) to[out=-120,in=15] (v3);
	    \draw[-] (w1) to[out=-60,in=165] (v2);
	    \draw[-] (w3) to[out=-120,in=15] (v1);
	    \draw[-] (w3) to[out=120,in=-15] (v6);
\end{tikzpicture}
 \end{minipage}
 \begin{minipage}[b]{0.49\textwidth}
 \centering\tiny
 \begin{tikzpicture}[scale = 1.2]
	    \foreach \i in {0,1,2,3,4,5} {
	    \node (v\i) at (0+\i/2,0.25) {};
	    \node (u\i) at (0+\i/2,1.75) {};
	    \node (w\i) at (0+\i/2,1) {};
	    }

	    \foreach \i in {0,2,4} {
	    \draw[-] (w\i) -- (v\i);
	    \draw[-,dashed] (w\i) -- (u\i);
	    }

	    \foreach \i in {1,3,5} {
	    \draw[-] (w\i) -- (v\i);
	    \draw[-,dashed] (w\i) -- (u\i);
	    }
	    
	    \foreach \i/\j in {0/1, 1/2, 2/3, 3/4, 4/5} {
	    \draw[-] (v\i) -- (v\j);
	    \draw[-] (u\i) -- (u\j);
	    }
	    
	    \node (z) at (1.25,1.375) {};
	    \foreach \i in {0,2,4} 
	    \draw[-] (z) -- (w\i);
	    
	    \node (x) at (1.75,0.625) {};
	    \foreach \i in {1,3,5} 
	    \draw[-] (x) -- (w\i);
	    
	    \draw[-] (v0) to[out=-22.5,in=-157.5] (v5);
	    \draw[-] (u0) to[out=22.5,in=157.5] (u5);

\end{tikzpicture}
 \end{minipage}
 \caption{The exceptional cubic graphs $H_1$ and $H_2$. The dashed edges illustrate a maximum uniquely restricted matching.
 Note that the dashed edges belong to the graph.} \label{fig_excub}
 \end{figure}
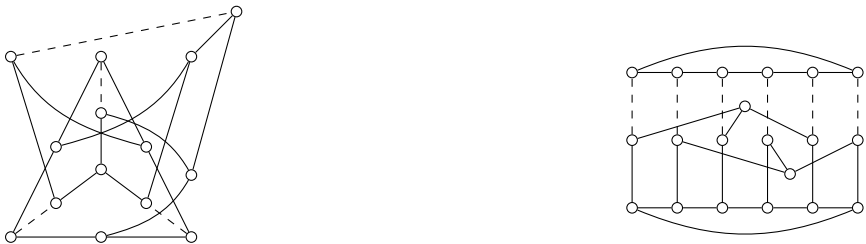 
Note that $H_1$ and $H_2$ have uniquely restricted matching number 
$4=\frac{n(H_1)-2}{3}$ and $6=\frac{n(H_2)-2}{3}$, respectively.
We prove the following.

\begin{theorem} \label{t1}
 If $G$ is a connected subcubic graph of girth at least $5$ that is not isomorphic to $H_1$ or $H_2$, then
 \begin{displaymath}
  \nu_{ur}(G) \geq \frac{n(G)-1}{3}.
 \end{displaymath}
\end{theorem}
In \cite{fhr} Theorem \ref{t1} was proved for graphs of girth at least $7$.
For this, the authors proved that any subcubic graph $G$ of girth at least $7$ 
that is not cubic and not a tree
satisfies
\begin{align} \label{motivation2}
 \nu_{ur}(G) \geq \frac{n(G)}{3}.
\end{align}
We will prove Theorem \ref{t1} with a similar approach.
However, (\ref{motivation2}) does not hold for graphs of girth at least $5$.
Consider, for example, the graph $G_1$ that arises from a $K_4$ by
subdividing each edge once.
Hence, we need to enlarge the family of graphs that we exclude.
For the definition of this family of graphs,
we recall the Gallai-Edmonds Structure Theorem, cf. \cite{lp}.

\begin{theorem}[Gallai-Edmonds Structure Theorem] \label{t2}
 Let $G$ be a graph, let $D_G$ be the set of vertices that are not covered by some maximum 
 matching in $G$, let $A_G = N_G(D_G) \setminus D_G$, and let $C_G = V(G) \setminus (A_G \cup D_G)$.
 The following statements hold.
 \begin{enumerate}
 \item[(i)] $D_{G-x} = D_G$ for every vertex $x$ in $A_G$.
  \item[(ii)] Each maximum matching in $G$ contains a perfect matching of $G[C_G]$,
  a maximum matching of each component of $G[D_G]$,
  and a matching connecting each vertex in $A_G$ to a vertex in $D_G$.
  \item[(iii)] Each component $H$ of $G[D_G]$ is factor-critical, 
  that is, $H-x$ contains a perfect matching for every $x$ in $V(H)$.
  \item[(iv)] $\nu(G) = \frac{1}{2} (n(G) + |A_G| - c(G[D_G])$,
  where $c(G[D_G])$ denotes the number of components of the graph $G[D_G]$.
 \end{enumerate}
\end{theorem}

Note that, for every tree $T$, the components of $G[D_T]$ are isolated vertices,
because every factor-critical graph is bridgeless.
Let $T$ be a tree such that each vertex in 
$A_T$ has degree at most $3$.
Let the forest $T'=T[A_T \cup D_T]$ have $\kappa$ components,
that is, $T'$ has $|A_T|+|D_T|-\kappa$ edges.
If $\kappa=1$, then $T'$ has at most $3|A_T|$ edges,
and, if $\kappa\geq 2$, then 
every component of $T'$ sends an edge to $C_T$ within $T$,
and $T'$ has at most $3|A_T|-\kappa$ edges.
It follows that $|D_T|\leq 2|A_T|+1$ 
with equality only if $\kappa=1$, 
every vertex in $A_T$ has degree $3$,
$A_T$ is independent, 
and $C_T$ is empty.
By Theorem \ref{t2}, we obtain
\begin{eqnarray}
 \nu(T) &\geq &\max\left\{|A_T|, \frac{1}{2} (n(T) + |A_T| - |D_T|)\right\} \label{ed1}\\
	&\geq &\max\left\{|A_T|, \frac{1}{2} (n(T) - |A_T| - 1)\right\} \label{ed2}\\
	&\geq &\frac{1}{3} |A_T| + \frac{2}{6} (n(T) - |A_T| - 1) \label{ed3}\\
	&= &\frac{n(T)-1}{3}.\nonumber
\end{eqnarray}
Let $\mathcal{T}$ be the set of all trees $T$
with matching number $\frac{n(T)-1}{3}$
such that each vertex in $A_T$ has degree at most $3$.
Note that there are infinitely many such trees, cf. e.g. \cite{hy}.

The set of graphs $\mathcal{B}$ contains $G_1$ and $8$ more graphs,
which will be defined in the next section.
At the moment, it is only important that those graphs
have maximum degree $3$, minimum degree $2$,
and between $2$ and $6$ vertices of degree $2$.
Let $H$ be a graph in $\mathcal{B}$, 
and let $u_1,\ldots,u_k$ be the vertices of degree $2$ in $H$.
For a tree $T$ in $\mathcal{T}$,
let $u$ be some vertex in $D_T$ with neighbors $v_1,\ldots,v_\ell$ in $A_T$,
and suppose that $k \geq \ell$.
We say that a graph $G$ arises from $T$
{\it by replacing $u$ by $H$}
if $G$ arises from the disjoint union of $T-u$
and $H$ by adding the edges 
$u_{i_1}v_1,\ldots,u_{i_\ell}v_\ell$
for some $\ell$ distinct indices $i_1,\ldots,i_\ell$ in $[k]$,
where $[k]$ denotes the set of positive integers at most $k$.

Let $\mathcal{G}$ be the set of graphs $G$
that arise from a tree $T$ in $\mathcal{T}$
by replacing some vertices in $D_T$ by graphs in $\mathcal{B}$,
see Figure \ref{fig_exexceptional}
for an illustration. 
We call $T$ the host tree of $G$.
Furthermore, let $\mathcal{G}_3$ be the set of cubic graphs of girth at least $5$.

\begin{figure}[H] 
 \centering\tiny
 \begin{tikzpicture}[scale = 1.2]
	    \node  (u1) at (0,0) {};
	    \node  (u2) at (2,0) {};
	    \node  (u3) at (1,2) {};
	    \node  (u4) at (1,0.75) {};
	    \node  (v1) at (1,0) {};
	    \node  (v2) at (1.5,1) {};
	    \node  (v3) at (0.5,1) {};
	    \node  (v4) at (0.5,0.375) {};
	    \node  (v5) at (1.5,0.375) {};
	    \node  (v6) at (1,1.375) {};
	    
	    \foreach \from/\to in {u1/v1, v1/u2, v2/u2, v2/u3, v3/u1, v3/u3, v4/u1, v4/u4, v5/u2, v5/u4, v6/u3, v6/u4}
	    \draw [-] (\from) -- (\to);
 
 \node (u) at (4.375,2) {};
 
 \foreach \i [evaluate=\i as \x using \i/2] in {0,1,...,5}
 \node (v\i) at (3.125+\x,1) {};
 
 \foreach \i in {0,1,...,5}
 \draw[-] (u) -- (v\i);
 
 \foreach \i [evaluate=\i as \x using \i/4] in {0,1,2,...,11}
 \node (u\i) at (3+\x,0) {};
 
 \foreach \i/\j in {0/0,0/1, 1/2, 1/3, 2/4, 2/5, 3/6, 3/7, 4/8, 4/9, 5/10, 5/11}
 \draw[-] (v\i) -- (u\j);
 
 \draw[-latex] (6.5,1) -- (8,1);

 	    \foreach \i in {0,1,2,3,4,5} {
	    \node (a\i) at (8.875+0.375/2+\i*0.75,2) {};	    
	    }
	    
	    \node (b0) at (8.875+0.375/2,1.33) {};
	    \node (b1) at (8.875+1.5+0.375/2,1.33) {};
	    \node (b2) at (8.875+3+0.375/2,1.33) {};	    
	    
	    \foreach \i/\j in {0/0,1/2,2/4} {
	    \draw[-] (b\i) -- (a\j);
	    }	    
	    
	    \foreach \i/\j in {0/1, 1/2, 2/3, 3/4, 4/5} {
	    \draw[-] (a\i) -- (a\j);
	    }
	    
	    \node (c) at (8.875+0.375/2+0.75+0.375,1.66) {};
	    \foreach \i in {0,1,2} 
	    \draw[-] (c) -- (b\i);
	    
	    \draw[-] (a0) to[out=22.5,in=157.5] (a5);
	    
	    \foreach \i [evaluate=\i as \x using \i*0.75] in {0,1,...,5}
	    \node (d\i) at (8.875+0.375/2+\x,0.66) {};
	    
	    \draw[-] (b0) -- (d0);
	    \draw[-] (a1) -- (d1);
	    \draw[-] (b1) -- (d2);
	    \draw[-] (a3) -- (d3);
	    \draw[-] (b2) -- (d4);
	    \draw[-] (a5) -- (d5);

	    \foreach \i [evaluate=\i as \x using \i*0.375] in {0,1,2,...,11}
	    \node (e\i) at (8.875+\x,0) {};

	    \foreach \i/\j in {0/0,0/1, 1/2, 1/3, 2/4, 2/5, 3/6, 3/7, 4/8, 4/9, 5/10, 5/11}
            \draw[-] (d\i) -- (e\j);

\end{tikzpicture}
 \caption{Some graphs in $\mathcal{B}$, $\mathcal{T}$, and $\mathcal{G}$.} \label{fig_exexceptional}
 \end{figure}
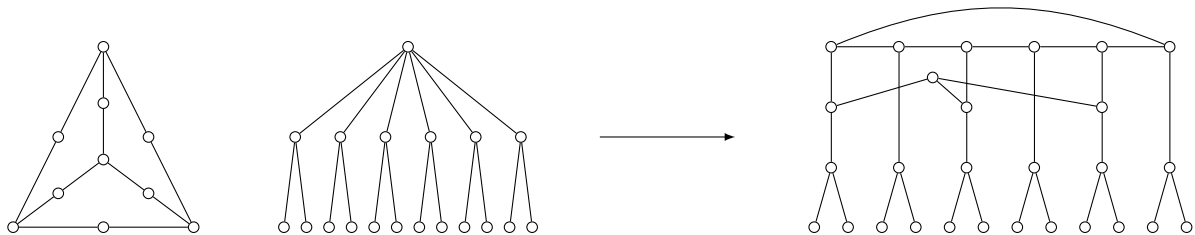 
 
For a set of connected graphs $\mathcal{H}$,
let $\kappa_{G}(\mathcal{H})$
be the number of components of 
$G$ isomorphic to a graph in $\mathcal{H}$. 
Theorem \ref{t1} is a consequence of the following more precise statement.

 \begin{theorem} \label{t3}
If $G$ is a subcubic graph of girth at least $5$, then
 \begin{displaymath}
  \nu_{ur}(G) \geq \frac{n(G)-\kappa_{G}(\mathcal{G}) - \kappa_{G}(\mathcal{G}_3) - \kappa_{G}(\{H_1,H_2\})}{3}.
 \end{displaymath}
\end{theorem}
Since every graph in $\mathcal{B}$ has at least two vertices of degree $2$,
Theorem \ref{t1} follows immediately from Theorem \ref{t3} and
$\mathcal{G} \cap \mathcal{G}_3 = \emptyset$.

We close the introduction with a few basic results and notations.
For a matching $M$ in a graph $G$, 
let $V(M)$ be the set of vertices that are covered by $M$.
Golumbic, Hirst, and Lewenstein \cite{ghl} observed that a matching $M$ in a graph $G$ 
is uniquely restricted 
if and only if there is no $M$-alternating cycle in $G$.
For a set $X \subseteq V(G)$, let 
$E_G(X) = \{uv \in E(G) \mid u \in X \mbox{ and } v \in V(G)\setminus X\}$,
and let $m_G(X) = |E_G(X)|$.

\section{The exceptional graph family}
We start with some properties of the trees in $\mathcal{T}$.

\begin{lemma} \label{l1}
If $T$ is in $\mathcal{T}$, then
$C_T = \emptyset$, 
$A_T$ is independent,
each vertex in $A_T$ has degree exactly $3$ in $T$,
and $\nu(T) = |A_T|$.
\end{lemma}
\begin{proof}
For a tree $T$ in $\mathcal{T}$,
the inequalities (\ref{ed1}), (\ref{ed2}), and (\ref{ed3})
hold with equality.
Equality in (\ref{ed3}) implies $|A_T|=\frac{1}{2}(n(T)-|A_T|-1)$,
and, hence, $|A_T|=\frac{n(T)-1}{3}=\nu(T)$.
Equality in (\ref{ed2}) implies $|D_T|=2|A_T|+1$,
which implies that 
every vertex in $A_T$ has degree $3$,
$A_T$ is independent, 
and $C_T$ is empty.
\end{proof}
By Lemma \ref{l1}, every tree $T$ in $\mathcal{T}$
has partite sets $A_T$ and $D_T$.

\begin{lemma} \label{l2}
If $T$ is in $\mathcal{T}$, and $X \subseteq D_T$, then the following statements hold.
 \begin{enumerate}
  \item[(i)] 	If $|X| \leq 2$, 
		then $\nu(T-X) = \nu(T)$.
  \item[(ii)] 	If $|X| = 3$ and no vertex in $V(T) \setminus X$ has $3$ neighbors in $X$,
		then $\nu(T-X) = \nu(T)$.
  \item[(iii)] 	If $|X| = 4$, no vertex in $V(T) \setminus X$ has $3$ neighbors in $X$,
		and no vertex $w$ in $V(T) \setminus X$ has neighbors $u$ and $v$ in $V(T) \setminus X$ with $X \subseteq N_T(\{u,v\})$,
		then $\nu(T-X) = \nu(T)$.
 \end{enumerate}
\end{lemma}
\begin{proof}
 Let $T^\prime = T - X$, and suppose that $T^\prime$ has matching number less than $\nu(T)=|A_T|$.
By Hall's Theorem, cf. \cite{lp}, this implies that there is a non-empty set $S \subseteq A_T$ such that $|N_{T^\prime}(S)| < |S|$.
 By Lemma \ref{l1}, $T[N_T[S]]$ is a forest with exactly $3|S|$ many edges, which implies that
 $|N_{T^\prime}(S)| \geq 2|S|+ 1 - |X|$, that is, $|S| \leq |X|-2$. This proves $(i)$.
 If $|S| = 1$, then $|N_{T^\prime}(S)| \geq 1$, because no vertex in $V(T) \setminus X$ is adjacent to $3$ vertices in $X$.
 This proves $(ii)$. Hence, we may assume that $|X| = 4$, and that $S$ contains two distinct vertices $u$ and $v$.
 Since $|N_{T}(S)| \geq 5$ and $|N_{T^\prime}(S)| \leq 1$, we obtain that $X \subseteq N_T(\{u,v\})$.
 Moreover, since neither $u$ nor $v$ have all its neighbors in $X$ 
 and $|N_{T^\prime}(S)| \leq 1$,
 it follows that $u$ and $v$ have a common neighbor $w$ in $V(T) \setminus X$, which is a contradiction. 
\end{proof}
The next two lemmas state that $\mathcal{T}$ is closed under contracting the edges
of the subgraphs corresponding to $(ii)$ and $(iii)$ of Lemma \ref{l2},
and that the Gallai-Edmonds decomposition is stable under these operations.

\begin{lemma}\label{l3}
Let $T$ be in $\mathcal{T}$. Let $v \in A_T$, and let $N_T(v) = \{u_1,u_2,u_3\}$.
If $T^\prime$ arises from $T$ by contracting the edges $vu_1$, $vu_2$, and $vu_3$,
and $w$ is the newly created vertex,
then $T^\prime \in \mathcal{T}$,
$D_{T^\prime} = \{w\} \cup D_T\setminus N_T(v)$, and
$A_{T^\prime} = A_T \setminus \{v\}$.
\end{lemma}
\begin{proof}
Note that, since $T$ is a tree, $T^\prime$ is a tree.
By Lemma \ref{l2} $(i)$, there is a maximum matching $M$ in $T$ that does not cover $u_1$ and $u_2$.
Since $vu_3 \in M$, the matching $M \setminus \{vu_3\}$ is a matching in $T^\prime$, which implies that
$\nu(T^\prime) \geq \nu(T) -1$. 
Every maximum matching $M$ in $T^\prime$ corresponds to a matching 
$M^\prime$ in $T$ of the same size that does not cover $v$ 
as well as at least two vertices in $N_T(v)$.
By symmetry, we may assume that $u_1$ is not covered by $M^\prime$, which, 
since $M^\prime \cup \{vu_1\}$ is a matching in $T$,
implies that $\nu(T) \geq \nu(T^\prime)+1$.
Thus, $T^\prime$ has matching number $\frac{n(T)-4}{3} = \frac{n(T^\prime)-1}{3}$.

Let $x \in \{w\} \cup D_T\setminus N_T(v)$. If $x \neq w$, then let $X = \{x,u_1\}$, 
and, if $x = w$, then let $X = \{u_1,u_2\}$.
By Lemma \ref{l2} $(i)$, $T$ has a maximum matching $M$ that does not cover
the vertices in $X$.
If $x \neq w$, then either $vu_2$ or $vu_3$ belongs to $M$, which implies that at most one edge in 
$E_T(N_T[v])$ belongs to $M$. 
By symmetry, we may assume that $vu_2 \in M$,
which implies that
$M \setminus \{vu_2\}$ corresponds 
to a maximum matching in $T^\prime$ of size $\nu(T^\prime)$ that does not cover $x$.
If $x = w$, then $vu_3$ belongs to $M$, and, similarly as before,
the matching $M \setminus \{vu_3\}$ is a matching in $T^\prime$ that does not cover $w$.
Thus, $\{w\} \cup D_T\setminus N_T(v) \subseteq D_{T^\prime}$.
Suppose that $y \in D_{T^\prime} \setminus (\{w\} \cup D_T\setminus N_T(v))$, and let
$M$ be some maximum matching in $T^\prime$ such that $y$ is not covered by $M$.
Then $M$ corresponds to a matching $M^\prime$ in $T$ of the same size that covers at most 
one vertex in $N_T(v)$, say $u_1$. Since the matching $M^\prime \cup \{vu_2 \}$
is maximum in $T$, and $y$ is not covered by $M$, it follows that $y \in D_T$,
which is a contradiction. 
Hence, $\{w\} \cup D_T\setminus N_T(v) = D_{T^\prime}$,
which implies that
$A_{T^\prime} = N_{T^\prime}(D_{T^\prime}) = A_T \setminus \{v\}$.
Thus, all vertices in $A_{T^\prime}$ have degree at most $3$,
which implies that $T^\prime \in \mathcal{T}$.
\end{proof}

\begin{lemma} \label{l4}
For $T \in \mathcal{T}$, let $v_1,v_2 \in A_T$, let 
$N_T(v_1) = \{u_1,u_2,u_5\}$, and let 
$N_T(v_2) = \{u_3,u_4,u_5\}$.
If $T^\prime$ arises from $T$ by contracting the edges in 
$E_T(\{v_1,v_2\})$,
and $w$ is the newly created vertex,
then $T^\prime \in \mathcal{T}$,
$D_{T^\prime} = \{w\} \cup D_T \setminus N_T(\{v_1,v_2\})$, and
$A_{T^\prime} = A_T \setminus \{v_1,v_2\}$.
\end{lemma}
The proof of Lemma \ref{l4} mimics the proof of Lemma \ref{l3}
and is therefore left to the reader.
Next, we define the graphs in $\mathcal{B}$, see Figures \ref{fig1}-\ref{fig3}.	
\begin{itemize}
 \item Let $G_1$ arise from $K_4$ by subdividing all edges once,
 \item let $G_2$ arise from $G_1$ by adding $3$ independent vertices and edges as illustrated in Figure \ref{fig1},
 \item let $G_3$ arise from $G_1$ by adding a $6$-cycle and edges as illustrated in Figure \ref{fig1}, and
 \item let $G_4$ arise from $G_3$ by adding $3$ independent vertices and edges as illustrated in Figure \ref{fig1}. 
\end{itemize}

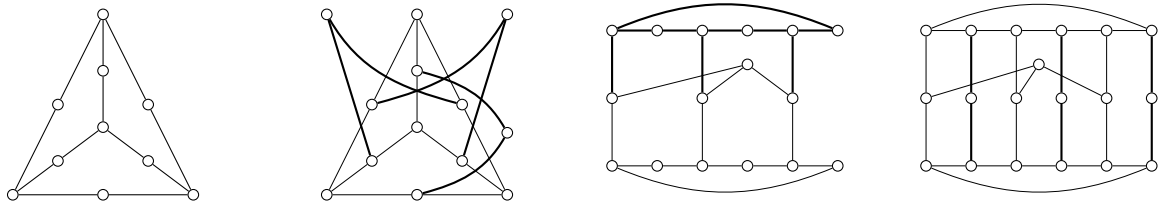
\begin{figure}[H] 
 \begin{minipage}[b]{0.25\textwidth}
 \centering\tiny
\begin{tikzpicture}[scale = 1.2]
	    \node  (u1) at (0,0) {};
	    \node  (u2) at (2,0) {};
	    \node  (u3) at (1,2) {};
	    \node  (u4) at (1,0.75) {};
	    \node  (v1) at (1,0) {};
	    \node  (v2) at (1.5,1) {};
	    \node  (v3) at (0.5,1) {};
	    \node  (v4) at (0.5,0.375) {};
	    \node  (v5) at (1.5,0.375) {};
	    \node  (v6) at (1,1.375) {};
	    
	    \foreach \from/\to in {u1/v1, v1/u2, v2/u2, v2/u3, v3/u1, v3/u3, v4/u1, v4/u4, v5/u2, v5/u4, v6/u3, v6/u4}
	    \draw [-] (\from) -- (\to);
\end{tikzpicture}
 \end{minipage}
\begin{minipage}[b]{0.24\textwidth}
 \centering\tiny
\begin{tikzpicture}[scale = 1.2]
	    \node (u1) at (0,0) {};
	    \node (u2) at (2,0) {};
	    \node (u3) at (1,2) {};
	    \node (u4) at (1,0.75) {};
	    \node (v1) at (1,0) {};
	    \node (v2) at (1.5,1) {};
	    \node (v3) at (0.5,1) {};
	    \node (v4) at (0.5,0.375) {};
	    \node (v5) at (1.5,0.375) {};
	    \node (v6) at (1,1.375) {};
	    
	    \node (w1) at (0,2) {};
	    \node  (w2) at (2,2) {};
	    \node  (w3) at (2,1.375/2) {};
	    
	    \foreach \from/\to in {u1/v1, v1/u2, v2/u2, v2/u3, v3/u1, v3/u3, v4/u1, v4/u4, v5/u2, v5/u4, v6/u3, v6/u4}
	    \draw [-, very thin] (\from) -- (\to);
	    
	    \foreach \from/\to in {w1/v4, w2/v5}
	    \draw [-, thick] (\from) -- (\to);
	    
	    \draw[-,thick] (w2) to[out=-120,in=15] (v3);
	    \draw[-,thick] (w1) to[out=-60,in=165] (v2);
	    \draw[-,thick] (w3) to[out=-120,in=15] (v1);
	    \draw[-,thick] (w3) to[out=120,in=-15] (v6);
\end{tikzpicture}
\end{minipage}
\begin{minipage}[b]{0.24\textwidth}
 \centering\tiny
\begin{tikzpicture}[scale = 1.2]

	    \foreach \i in {0,1,2,3,4,5} {
	    \node (v\i) at (0+\i/2,0.25) {};
	    \node (u\i) at (0+\i/2,1.75) {};	    
	    }
	    
	    \node (w0) at (0,1) {};
	    \node (w1) at (1,1) {};
	    \node (w2) at (2,1) {};

	    \foreach \i/\j in {0/0,1/2,2/4} {
	    \draw[-, very thin] (w\i) -- (v\j);
	    \draw[-,thick] (w\i) -- (u\j);
	    }

	    \foreach \i/\j in {0/1, 1/2, 2/3, 3/4, 4/5} {
	    \draw[-, very thin] (v\i) -- (v\j);
	    \draw[-,thick] (u\i) -- (u\j);
	    }
	    
	    \node (z) at (1.5,1.375) {};
	    \foreach \i in {0,1,2} 
	    \draw[-, very thin] (z) -- (w\i);
	    
	    \draw[-, very thin] (v0) to[out=-22.5,in=-157.5] (v5);
	    \draw[-, thick] (u0) to[out=22.5,in=157.5] (u5);
	    
\end{tikzpicture}
\end{minipage}
\begin{minipage}[b]{0.25\textwidth}
 \centering\tiny
\begin{tikzpicture}[scale = 1.2]

	    \foreach \i in {0,1,2,3,4,5} {
	    \node (v\i) at (0+\i/2,0.25) {};
	    \node (u\i) at (0+\i/2,1.75) {};
	    \node (w\i) at (0+\i/2,1) {};
	    }

	    \foreach \i in {0,2,4} {
	    \draw[-, very thin] (w\i) -- (v\i);
	    \draw[-, very thin] (w\i) -- (u\i);
	    }
	    
	    \foreach \i in {1,3,5} {
	    \draw[-,thick] (w\i) -- (v\i);
	    \draw[-,thick] (w\i) -- (u\i);
	    }
	    
	    \foreach \i/\j in {0/1, 1/2, 2/3, 3/4, 4/5} {
	    \draw[-, very thin] (v\i) -- (v\j);
	    \draw[-, very thin] (u\i) -- (u\j);
	    }
	    
	    \node (z) at (1.25,1.375) {};
	    \foreach \i in {0,2,4} 
	    \draw[-, very thin] (z) -- (w\i);
	    
	    \draw[-, very thin] (v0) to[out=-22.5,in=-157.5] (v5);
	    \draw[-, very thin] (u0) to[out=22.5,in=157.5] (u5);
	    
\end{tikzpicture}
\end{minipage}
\caption{The graphs $G_1$, $G_2$, $G_3$, and $G_4$. The thick edges illustrate the added part.} \label{fig1}
\end{figure} 

All remaining graphs $G_5,\ldots,G_9$ arise from $G_2$, $G_3$, or $G_4$ either by adding a path of length $2$ 
or a cycle of length $6$ and connecting them to the existing graph by some additional edges, 
see Figures \ref{fig2} and \ref{fig3}.
\begin{figure}[H] 
 \begin{minipage}[b]{0.49\textwidth}
 \centering\tiny
 \begin{tikzpicture}[scale = 1.2]
	    \node (u1) at (0,0) {};
	    \node (u2) at (2,0) {};
	    \node (u3) at (1,2) {};
	    \node (u4) at (1,0.75) {};
	    \node (v1) at (1,0) {};
	    \node (v2) at (1.5,1) {};
	    \node (v3) at (0.5,1) {};
	    \node (v4) at (0.5,0.375) {};
	    \node (v5) at (1.5,0.375) {};
	    \node (v6) at (1,1.375) {};
	    
	    \node(x1)[label=above:\footnotesize$v$] at (1,2.75) {};
	    \node(x2)[label=right:\footnotesize$u$] at (2.5,2) {};
	    \node(x3)[label=right:\footnotesize$w$] at (2.5,5.375/4) {};

	    \node  (w1) at (0,2) {};
	    \node  (w2) at (2,2) {};
	    \node  (w3) at (2,1.375/2) {};
	    
	    \foreach \from/\to in {u1/v1, v1/u2, v2/u2, v2/u3, v3/u1, v3/u3, v4/u4, v5/u4, v6/u4}
	    \draw [-, very thin] (\from) -- (\to);
	    
	    \foreach \from/\to in {w1/v4, w2/v5}
	    \draw [-, very thin] (\from) -- (\to);
	    
	    \foreach \from/\to in {x1/w2, x2/x3, x3/w3}
	    \draw [-, thick] (\from) -- (\to);
	    
	    \draw[-,thick] (x1) to[out=0,in=120] (x2);
	    \draw[-, very thin] (w2) to[out=-120,in=15] (v3);
	    \draw[-, very thin] (w1) to[out=-60,in=165] (v2);
	    \draw[-, dashed, very thin] (w3) to[out=-120,in=15] (v1);
	    \draw[-, very thin] (w3) to[out=120,in=-15] (v6);
	    
	    \foreach \from/\to in {u1/v4, u2/v5, u3/v6}
	    \draw [-, dashed, very thin] (\from) -- (\to);
	    
	     \draw[-, dashed, thick] (x1) -- (w1);
\end{tikzpicture}
 \end{minipage}
 \begin{minipage}[b]{0.49\textwidth}
 \centering\tiny
 \begin{tikzpicture}[scale = 1.2]
	    \node (u1) at (0,0) {};
	    \node (u2) at (2,0) {};
	    \node (u3) at (1,2) {};
	    \node (u4) at (1,0.75) {};
	    \node (v1) at (1,0) {};
	    \node (v2) at (1.5,1) {};
	    \node (v3) at (0.5,1) {};
	    \node (v4) at (0.5,0.375) {};
	    \node (v5) at (1.5,0.375) {};
	    \node (v6) at (1,1.375) {};
	    
	    \foreach \i in {0,1,2,3,4,5}
	    \node (x\i) at (3,2-0.4*\i) {};
	    
	    \foreach \i/\j in {0/1, 1/2, 2/3, 3/4, 4/5} {
	    \draw[-, thick] (x\i) -- (x\j);
	    }
	        
	    \node  (w1) at (0,2) {};
	    \node  (w2) at (2,2) {};
	    \node  (w3) at (2,1.375/2) {};
	    
	    \foreach \from/\to in {u1/v1, v1/u2, v2/u2, v2/u3, v3/u1, v3/u3, v4/u4, v5/u4, v6/u4}
	    \draw [-, very thin] (\from) -- (\to);
	    
	    \foreach \from/\to in {w1/v4, w2/v5}
	    \draw [-, very thin] (\from) -- (\to);

	    \draw[-, very thin] (w2) to[out=-120,in=15] (v3);
	    \draw[-, very thin] (w1) to[out=-60,in=165] (v2);
	    \draw[-, dashed, very thin] (w3) to[out=-120,in=15] (v1);
	    \draw[-, very thin] (w3) to[out=120,in=-15] (v6);
	    
	    \draw[-, dashed, thick] (x0) to[out=150,in=30] (w1);
	    \draw[-, dashed, thick] (x2) -- (w2);
	    \draw[-, thick] (x4) -- (w3);
	    \draw[-, thick] (x5) to[out=67.5,in=-67.5] (x0);
	    \foreach \from/\to in {u1/v4, u2/v5, u3/v6}
	    \draw [-, dashed, very thin] (\from) -- (\to);

\end{tikzpicture}
 \end{minipage}
 \caption{The graphs $G_5$ and $G_6$. The dashed edges illustrate a maximum uniquely restricted matching.
 Note that the dashed edges belong to the graph.} \label{fig2}
 \end{figure}
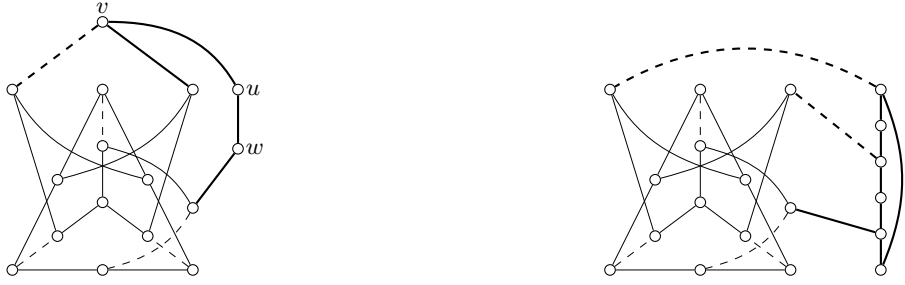 
 \begin{figure}[H] 
 \begin{minipage}[b]{0.33\textwidth}
 \centering \tiny
\begin{tikzpicture}[scale = 1.2]
	  \foreach \i in {0,1,2,3,4,5} {
	    \node (v\i) at (0+\i/2,0.25) {};
	    \node (u\i) at (0+\i/2,1.75) {};	    
	    }
	    
	    \node (w0) at (0,1) {};
	    \node (w1) at (1,1) {};
	    \node (w2) at (2,1) {};
	    
	    \node (x0) at (3.5,1) {};
	    \node (x1) at (3,1.375) {};
	    \node (x2) at (3,1.25/2) {};
	    
	    \draw[-, thick] (x0) -- (x1);
	    \draw[-, thick] (x0) -- (x2);
	    
	    \draw[-, thick, dashed] (x1) -- (v3);
	    \draw[-, thick] (x1) -- (u5);
	    \draw[-, thick] (x2) -- (u3);
	    \draw[-, thick, dashed] (x2) -- (v5);
	    
	    \foreach \i/\j in {2/4} {
	    \draw[-, very thin] (w\i) -- (v\j);
	    \draw[-,very thin] (w\i) -- (u\j);
	    }
	    \draw[-, very thin] (w0) -- (u0);
	    \draw[-, very thin, dashed] (w0) -- (v0);
	    \draw[-, very thin] (w1) -- (u2);
	    \draw[-, very thin, dashed] (w1) -- (v2);

	    \foreach \i/\j in {0/1, 1/2, 3/4, 4/5} {
	    \draw[-, very thin] (v\i) -- (v\j);
	    \draw[-,very thin] (u\i) -- (u\j);
	    }
	    \draw[-, very thin, dashed] (u2) -- (u3);
	    \draw[-, very thin] (v2) -- (v3);
	    
	    \node (z) at (1.5,1.375) {};
	    \foreach \i in {0,1} 
	    \draw[-, very thin] (z) -- (w\i);
	    
	    \draw[-, very thin, dashed] (z) -- (w2);
	    
	    \draw[-, very thin] (v0) to[out=-22.5,in=-157.5] (v5);
	    \draw[-, very thin] (u0) to[out=22.5,in=157.5] (u5);
	    
\end{tikzpicture}
\end{minipage}
 \begin{minipage}[b]{0.33\textwidth}
 \centering\tiny
 \begin{tikzpicture}[scale = 1.2]
	    \foreach \i in {0,1,2,3,4,5} {
	    \node (v\i) at (0+\i/2,0.25) {};
	    \node (u\i) at (0+\i/2,1.75) {};
	    \node (w\i) at (0+\i/2,1) {};
	    }
	    \node (x1)[label=above:\footnotesize$v$] at (1,2.3) {};
	    \node (x2)[label=above:\footnotesize$u$] at (1.5,2.3) {};
	    \node (x3)[label=above:\footnotesize$w$] at (2,2.3) {};
	    
	    \foreach \i/\j in {1/2,2/3} {
	    \draw[-, thick] (x\i) -- (x\j);
	    }

	    \foreach \i in {1,2,3} {
	    \draw[-, very thin] (w\i) -- (v\i);
	    \draw[-, very thin] (w\i) -- (u\i);
	    }
	    \draw[-, very thin] (u0) -- (w0);
	    \draw[-, very thin, dashed] (v0) -- (w0);
	    
	    \draw[-, very thin] (u4) -- (w4);
	    \draw[-, very thin, dashed] (v4) -- (w4);
	    
	    \draw[-, very thin, dashed] (u5) -- (w5);
	    \draw[-, very thin] (v5) -- (w5);
	    
	    \foreach \i/\j in {1/2, 3/4, 4/5} {
	    \draw[-, very thin] (v\i) -- (v\j);
	    \draw[-, very thin] (u\i) -- (u\j);
	    }
	    
	    \draw[-, very thin, dashed] (u0) -- (u1);
	    \draw[-, very thin, dashed] (u2) -- (u3);
	    
	    \draw[-, very thin] (v0) -- (v1);
	    \draw[-, very thin, dashed] (v2) -- (v3);

	    \node (z) at (1.25,1.375) {};
	    \foreach \i in {0,2,4} 
	    \draw[-, very thin] (z) -- (w\i);
	    
	    \draw[-, very thin] (v0) to[out=-22.5,in=-157.5] (v5);
	    \draw[-, very thin] (u0) to[out=22.5,in=157.5] (u5);
	    
	    \draw[-, thick, dashed] (x1) -- (w1);
	    \draw[-, thick] (x1) -- (w3);
	    \draw[-, thick] (x3) -- (w5);
\end{tikzpicture}
 \end{minipage}
 \begin{minipage}[b]{0.33\textwidth}
 \centering\tiny
 \begin{tikzpicture}[scale = 1.2]
	    \foreach \i in {0,1,2,3,4,5} {
	    \node (v\i) at (0+\i/2,0.25) {};
	    \node (u\i) at (0+\i/2,1.75) {};
	    \node (w\i) at (0+\i/2,1) {};
	    \node (x\i) at (0+\i/2,2.3) {};
	    }

	    \foreach \i in {1,2,3,4} {
	    \draw[-, very thin] (w\i) -- (v\i);
	    \draw[-, very thin] (w\i) -- (u\i);
	    }
	    \draw[-, very thin, dashed] (v0) -- (w0);
	    \draw[-, very thin] (u0) -- (w0);
	    
	    \draw[-, very thin] (v5) -- (w5);
	    \draw[-, very thin, dashed] (u5) -- (w5);

	    \foreach \i/\j in {0/1, 2/3, 4/5} {
	    \draw[-, very thin] (u\i) -- (u\j);
	    }
	    \draw[-, very thin, dashed] (u1) -- (u2);
	    \draw[-, very thin, dashed] (u3) -- (u4);
	    
	    \foreach \i/\j in {0/1, 1/2, 3/4, 4/5} {
	    \draw[-, very thin] (v\i) -- (v\j);
	    }
	    \draw[-, very thin, dashed] (v2) -- (v3);
	    
	    \foreach \i/\j in {0/1, 1/2, 2/3, 3/4, 4/5} {
	    \draw[-, thick] (x\i) -- (x\j);
	    }

	    \node (z) at (1.25,1.375) {};
	    \foreach \i in {0,2} 
	    \draw[-, very thin] (z) -- (w\i);
	    
	    \draw[-, very thin, dashed] (z) -- (w4);
	    
	    \draw[-, very thin] (v0) to[out=-22.5,in=-157.5] (v5);
	    \draw[-, very thin] (u0) to[out=22.5,in=157.5] (u5);
	    \draw[-, thick] (x0) to[out=22.5,in=157.5] (x5);
	    
	    \draw[-, thick, dashed] (x0) -- (w1);
	    \draw[-, thick, dashed] (x2) -- (w3);
	    \draw[-, thick] (x4) -- (w5);
\end{tikzpicture}
 \end{minipage}
 \caption{The graphs $G_7$, $G_8$, and $G_9$.} \label{fig3}
 \end{figure}
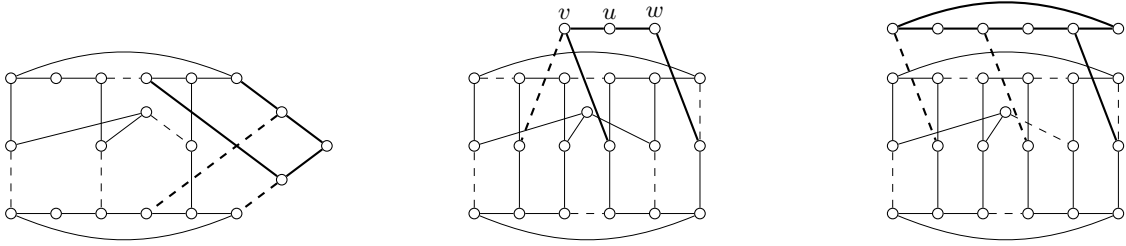
Let $\mathcal{B} = \{G_1,\ldots,G_9 \}$, and let $\mathcal{B}_c = \{G_1,\ldots,G_4\}$.
The following table collects some relevant information on the graphs in $\mathcal{B}$.
It also captures whether one of the following statements holds for a graph $H$ in $\mathcal{B}$.
 \begin{enumerate}
  \item[$(i)$] 	
$\nu_{ur}(H-X) = \nu_{ur}(H)$
for every set $X \subseteq \{u \in V(H) \mid d_G(u) = 2\}$
with $|X| \leq 2$.
  \item[$(ii)$] 	
$\nu_{ur}(H-X) = \nu_{ur}(H)$
for every set $X \subseteq \{u \in V(H) \mid d_G(u) = 2\}$
such that $|X| = 3$ and no vertex in $V(H) \setminus X$ has $3$ neighbors in $X$.
  \item[$(iii)$] 	
$\nu_{ur}(H-X) = \nu_{ur}(H)$
for every set $X \subseteq \{u \in V(H) \mid d_G(u) = 2\}$
such that $|X| = 4$, no vertex in $V(H) \setminus X$ has $3$ neighbors in $X$,
and no vertex $w$ in $V(H) \setminus X$ 
has neighbors $u$ and $v$ in $V(H) \setminus X$ with $X \subseteq N_H(\{u,v\})$.
 \end{enumerate}
While it is easy to see whether the statements $(i)$-$(iii)$ hold for a graph in $\mathcal{B}$,
we determined their uniquely restricted matching number using a computer.
\begin{table}[H]
\begin{center} 
\begin{tabular}{c|c|c|c|c|c}
 & $n(H)$ & $\nu_{ur}(H)$ & $(i)$ & $(ii)$ & $(iii)$ \\ \hline
 $G_1$ & 10 & 3 & \cmark & \cmark & -\\
 $G_2$ & 13 & 4 & \cmark & \xmark & - \\ 
 $G_3$ & 16 & 5 & \cmark & \cmark & \cmark \\ 
 $G_4$ & 19 & 6 & \cmark & \xmark & - \\ 
 $G_5$ & 16 & 5 & \cmark & - & - \\ 
 $G_6$ & 19 & 6 & \cmark & \cmark & - \\ 
 $G_7$ & 19 & 6 & \cmark & \cmark & - \\ 
 $G_8$ & 22 & 7 & \cmark & - & - \\ 
 $G_9$ & 25 & 8 & \cmark & \cmark & - \\ 
\end{tabular}
\end{center}
\caption{The sign '\xmark' means that such a set $X$ exists but the conclusion is not true
while the sign '-' means that there is no such set $X$.} \label{table1}
\end{table} 
Before we proceed, we give some notation which will be used in the remaining
lemmas of this section as well as in the proof of Theorem \ref{t3}.
For a graph $G$ in $\mathcal{G}$, 
let $T$ be the host tree of $G$,
and let $H_1,\ldots,H_k$ be the blocks of $G$ isomorphic to a graph in $\mathcal{B}$.
Let $I \subseteq [k]$, let $G^\prime$ arise from
$G$ by contracting the edges of $H_i$ for every $i$ in $I$,
and let $v_i$ be the vertex in $G^\prime$ corresponding to $H_i$ for every $i$ in $I$.
Note that, if $I = [k]$, then $G^\prime$ is isomorphic to $T$.
For the host tree $T$, let $D_T^\prime = D_T \setminus \{v_1,\ldots,v_k\}$.
Each vertex of $G$ corresponds to a vertex in $G^\prime$;
namely if $u$ is some vertex in $G$, then $u$ either belongs to 
$(A_T \cup D_T) \setminus \{v_i \mid i \in I\} \subseteq V(G^\prime)$, or
$u$ belongs to some block $H_i$, in which case we say that 
$u$ corresponds to $v_i$.

For some uniquely restricted matching $M$ in $G^\prime$, 
we call $M^\prime$ the corresponding 
matching in $G$, if $M^\prime$ arises from $M$
by replacing the edges $uv_i$ in $M$ by $uv$ for some $v$ in $V(H_i)$
that is adjacent to $u$. Since all edges in 
$E(G) \setminus \bigcup_{i=1}^k{E(H_i)}$ are bridges, 
the matching $M^\prime$ is uniquely restricted in $G$.
\begin{lemma} \label{l5}
Every graph $G$ in $\mathcal{G}$ satisfies $\nu_{ur}(G) = \frac{n(G)-1}{3}$.
\end{lemma}
\begin{proof}
 Let $M$ be a maximum uniquely restricted matching in $G$,
 and let $M^\prime = M \cap \left(\bigcup_{i=1}^k{E(H_i)}\right)$.
 Since the graph $G^\prime = G - \bigcup_{i=1}^k{E(H_i)}$ is bipartite 
 with partition classes $A_T$ and $V(G) \setminus A_T$, 
 we have that $|M\setminus M^\prime| \leq |A_T|$.
 Furthermore, we have that $|M^\prime| \leq \sum_{i=1}^k{\nu_{ur}(H_i)}$.
 By Lemma \ref{l1} and Table \ref{table1}, we obtain that
 \begin{align*}
  \nu_{ur}(G) \leq |A_T| + \sum_{i=1}^k{\nu_{ur}(H_i)} 
	      =\frac{n(T)-1}{3} + \sum_{i=1}^k{\frac{n(H_i)-1}{3}} 
	      =\frac{n(G)-1}{3}.
 \end{align*}
Let $M$ be some maximum matching in $T$,
and let $M^\prime$
be the corresponding 
uniquely restricted matching in $G$.
Since $|V(M^\prime) \cap V(H_i)| \leq 1$ for each $i \in [k]$,
$H_i$ contains a uniquely restricted matching $M_i$
of size $\frac{n(H_i)-1}{3}$ that is disjoint from $M^\prime$, see Table \ref{table1}.
Since all edges in $E(G) \setminus \bigcup_{i=1}^k{E(H_i)}$ are bridges in $G$,
the matching $M^\prime \cup \bigcup_{i=1}^k{M_i}$ is uniquely restricted in $G$.
Hence, we obtain that
\begin{align*}
 \nu_{ur}(G) \geq \left| M^\prime \cup \bigcup_{i=1}^k{M_i} \right| 
	       = \frac{n(T)-1}{3} + \sum_{i=1}^k{\frac{n(H_i)-1}{3}} 
	       = \frac{n(G)-1}{3}.
\end{align*}
\end{proof}
Our next lemma shows that Lemma \ref{l2} $(i)$ essentially holds for 
the graphs in $\mathcal{G}$.
\begin{lemma} \label{l6}
Let $G \in \mathcal{G}$, and let $X \subseteq D_{T}^\prime \cup \{u \in V(G) \mid d_G(u) = 2\}$.
If $|X| \leq 2$, then $\nu_{ur}(G-X) = \nu_{ur}(G)$.
\end{lemma}
\begin{proof}
Let $X^\prime = X \cap D_T^{\prime}$, and let 
$X_T = X^\prime \cup \{v_i \mid i \in [k] \mbox{ and } X \cap V(H_i) \neq \emptyset \}$.
Since $|X| \leq 2$, it follows that $|X_T| \leq 2$, 
which, by Lemma \ref{l2} $(i)$, implies that $T-X_T$ contains a matching $M$ of size $\nu(T) = \frac{n(T)-1}{3}$.
Let $M^\prime$ be the corresponding 
uniquely restricted matching in $G-X$.
Let $i \in [k]$. If $V(M^\prime) \cap V(H_i) = \emptyset$, then
there is a uniquely restricted matching $M_i$ of size 
$\frac{n(H_i)-1}{3}$ in $H_i - X$ since $|X| \leq 2$.
If $V(M^\prime) \cap V(H_i) \neq \emptyset$, then 
$X \cap V(H_i) = \emptyset$ and $|V(M^\prime) \cap V(H_i)| = 1$, which
implies that there is a uniquely restricted matching $M_i$ of size 
$\frac{n(H_i)-1}{3}$ in $H_i$
that is disjoint from $M^\prime$, see Table \ref{table1}.
Since 
the matching $M^\prime \cup \bigcup_{i=1}^k{M_i}$ is
uniquely restricted in $G-X$, 
we obtain that
\begin{align*}
 \nu_{ur}(G-X) \geq \frac{n(T)-1}{3} + \sum_{i=1}^k{\frac{n(H_i)-1}{3}} 
	       = \frac{n(G)-1}{3}.
\end{align*}
\end{proof}
Let $\mathcal{G}_c$ be the set of graphs $G$ in $\mathcal{G}$
such that no block of $G$ belongs to $\mathcal{B} \setminus \mathcal{B}_c$.
Our final lemma in this section shows that Lemma \ref{l2} $(ii)$ essentially holds for 
the graphs in $\mathcal{G}_c$.
\begin{lemma} \label{l7}
 Let $G \in \mathcal{G}_c$ be such that it is not isomorphic to $G_2$ or $G_4$, 
 and let $X \subseteq D_{T}^\prime \cup \{u \in V(G) \mid d_G(u) = 2\}$.
 If $|X| = 3$ and no vertex in $V(G) \setminus X$ has $3$ neighbors in $X$,
 then $\nu_{ur}(G-X) = \nu_{ur}(G)$.
\end{lemma}
\begin{proof}
Let $X = \{x_1,x_2,x_3\}$.
If $X \subseteq D_{T}^\prime$, then, by Lemma \ref{l2} $(ii)$,
$T$ contains a maximum matching $M$ of
size $\frac{n(T)-1}{3}$ 
that does not cover any vertex in $X$.
Let $M^\prime$ be the corresponding
uniquely restricted matching in $G-X$.
Let $i \in [k]$.
Since $|V(M^\prime) \cap V(H_i)| \leq 1$, 
the subgraph $H_i$ of $G$ contains a  
uniquely restricted matching $M_i$ 
of size $\frac{n(H_i)-1}{3}$
such that $M_i$ is disjoint from $M^{\prime}$,
see Table \ref{table1}.
Since the matching 
$M^{\prime} \cup \bigcup_{i=1}^k{M_i}$
is uniquely restricted in $G-X$,
it follows that
\begin{align*}
  \nu_{ur}(G-X) \geq \frac{n(T)-1}{3} + \sum_{i=1}^k \frac{n(H_i)-1}{3} 
	      = \frac{n(G)-1}{3}. 
\end{align*}
Now, we assume that $X \not\subseteq D_T^\prime$.
Let $i$ be such that it minimizes $|V(H_i) \cap X|$ among
all non-empty intersections, and let $G^\prime$ arise from
$G$ by contracting the edges in $H_i$.
First, we assume that 
$|V(H_i) \cap X| = 1$, say $x_1 \in V(H_i)$.
By Lemma \ref{l6}, $G^{\prime}$ contains a
uniquely restricted matching $M$ of size at least
$\frac{n(G^{\prime})-1}{3}$
that does not cover $x_2$ and $x_3$.
Let $M^\prime$ be the corresponding 
uniquely restricted matching in $G-X$.
Since $|V(M^\prime) \cap V(H_i)| \leq 1$,
it follows that
$H_i$ contains a 
uniquely restricted matching $M_i$
of size
$\frac{n(H_i)-1}{3}$
that does not cover $x_1$ and is disjoint from $M^\prime$, see Table \ref{table1}.
Since the matching 
$M^{\prime} \cup M_i$ is
uniquely restricted in $G-X$, it follows that
\begin{align*}
 \nu_{ur}(G-X) \geq \frac{n(G^\prime)-1}{3} + \frac{n(H_i)-1}{3} 
	      = \frac{n(G)-1}{3}.
\end{align*}
Hence, we may assume that $|V(H_i) \cap X| \geq 2$.
By Lemma \ref{l6}, $G^{\prime}$ contains a
uniquely restricted matching $M$ of size at least
$\frac{n(G^{\prime})-1}{3}$
that does not cover $X \setminus V(H_i)$ and $v_i$,
because $|X \setminus V(H_i)| \leq 1$.
Note that $V(H_i) \cap V(M) = \emptyset$.
If $|V(H_i) \cap X| = 2$, then $H_i$ contains 
a matching $M_i$
of size $\frac{n(H_i)-1}{3}$
that does not cover the vertices in $X \cap V(H_i)$, see Table \ref{table1}.
Next, we assume that $|V(H_i) \cap X| = 3$. 
Every vertex in $V(H_i)$ that is adjacent to a vertex
in $A_T$ has degree $3$ in $G$.
Since every vertex in $X \cap V(H_i)$ has degree $2$ in $G$,
this implies that that no vertex in $X \cap V(H_i)$
is adjacent to a vertex in $A_T$.
If $H_i$ is isomorphic to $G_2$ or $G_4$,
then $H_i$ contains exactly $3$ vertices of degree $2$,
which must all belong to $X$. Therefore,
no vertex in $V(H_i)$ is adjacent to a vertex in $A_T$,
which, since $G$ is connected, implies that
$G$ is isomorphic to $H_i$, a contradiction. 
Hence, $H_i$ is isomorphic to $G_1$ or $G_3$, 
and thus contains a 
uniquely restricted matching $M_i$
of size
$\frac{n(H_i)-1}{3}$
that does not cover the vertices in $X \cap V(H_i)$, see Table \ref{table1}.
Since $V(M) \cap V(H_i) = \emptyset$, it follows that
$M \cup M_i$ is a uniquely restricted matching
in $G-X$ and of size at least $\frac{n(G)-1}{3}$.
\end{proof}

\section{Proof of Theorem \ref{t3}}
\begin{proof}[Proof of Theorem \ref{t3}]
Suppose, for a contradiction, that the theorem is false, and let 
$G$ be a counterexample of minimum order. $G$ is connected,
 not in $\mathcal{G}$, and not isomorphic to $H_1$ or $H_2$.
 Therefore, any proper induced subgraph $H$ of $G$ satisfies
 \begin{align*}
  \nu_{ur}(H) \geq \frac{n(G)-\kappa_H(\mathcal{G})}{3}.
 \end{align*}
 \begin{claim} \label{c1}
 Any induced subgraph $H$ of $G$ that belongs to $\mathcal{G}$
 satisfies $m_G(V(H)) \geq 2$.
\end{claim}
\begin{proof}
 Let $H \in \mathcal{G}$ be an induced subgraph of $G$
 and suppose, for a contradiction, that $uv$ is the only edge in $E_G(V(H))$ such that
 $u \in V(H)$.
 Let $G^\prime = G-(V(H) \cup \{v\})$ and suppose that at most one 
 component of $G^\prime$ is in $\mathcal{G}$.
 By Lemma \ref{l6}, $H$ contains a uniquely restricted matching $M_H$
 of size $\frac{n(H)-1}{3}$ that does not cover $u$.
 Furthermore, $G^\prime$ contains a uniquely restricted matching $M^\prime$
 of size at least $\frac{n(G^\prime)-1}{3}$.
 Since the matching $M_H \cup M^\prime \cup \{uv\}$ is uniquely restricted in $G$,
 it follows that
 \begin{align*}
  \nu_{ur}(G) \geq \frac{n(H)+n(G^\prime)-2}{3} + 1
	      = \frac{n(G)}{3}, 
 \end{align*}
which is a contradiction. 
Hence, we may assume that 
$G^\prime$ has two components, and that	
both components 
$H_1$ and $H_2$ of $G^\prime$ belong to $\mathcal{G}$.
Let $T_H$, $T_1$, and $T_{2}$ be the host trees of $H$, $H_1$, and $H_2$, respectively.
Let $w_H$, $w_1$, and $w_2$ be the vertices in
$T_H$, $T_1$, and $T_2$ such that their corresponding vertices in $G$
are adjacent to $v$.
Let $T$ arise from $T_H \cup T_1 \cup T_2$ by adding $v$
along with the edges $vw_H$, $vw_1$, and $vw_2$.
Lemma \ref{l1} implies that $w_H$ belongs to $D_{T_H}$, which implies that 
$v$ belongs to any maximum matching in $T$.
Hence, Theorem \ref{t2} $(i)$ implies that $D_{T} = D_{T_H} \cup D_{T_1} \cup D_{T_2}$.
In particular, $G$ arises from $T$ by replacing some vertices of $D_T$ by 
graphs in $\mathcal{B}$, which implies that $G \in \mathcal{G}$, a contradiction.
\end{proof}
\begin{claim} \label{c2}
Any induced subgraph $H$ of $G$ that belongs to $\mathcal{G}$
satisfies $m_G(V(H)) \geq 3$ if $n(H) \geq 2$.
\end{claim}
\begin{proof}
 Let $H \in \mathcal{G}$ be an induced subgraph of $G$ with order at least $2$ 
 and suppose, for a contradiction, that
 $uv$ and $wx$ are the only two edges in $E_G(V(H))$ such that
 $u, w \in V(H)$. 
 First, we assume that $uv$ and $wx$ are disjoint.
 Let $G^\prime = G-(V(H) \cup \{v\})$.
 By Lemma \ref{l6}, $H$ contains a uniquely restricted matching $M_H$
 of size $\frac{n(H)-1}{3}$ that does not cover $u$ and $w$.
 Since $m_G(V(H) \cup \{v\}) \leq 3$, Claim \ref{c1} implies that
 at most one component of $G^\prime$ 
 belongs to $\mathcal{G}$,
 which, by the choice of $G$,
 implies that $G$
 contains a uniquely restricted matching $M^\prime$
 of size at least $\frac{n(G^\prime)-1}{3}$.
 Since the matching $M_H \cup M^\prime \cup \{uv\}$ is uniquely restricted in $G$,
 it follows that
 \begin{align*}
  \nu_{ur}(G) \geq \frac{n(H)+n(G^\prime)-2}{3} + 1
	      = \frac{n(G)}{3}, 
 \end{align*}
which is a contradiction.
Hence, we may assume that $uv$ and $wx$ share a common vertex.
First, we assume that $u=w$. Let $G^\prime = G-(V(H) \setminus \{u\})$.
Since $m_G(V(H)\setminus \{u\}) = 1$, it follows, by Claim \ref{c1}, that neither
$G^\prime$ nor $H-u$ belong to $\mathcal{G}$, which implies that
$G^\prime$ and $H-u$ contain uniquely restricted matchings $M^\prime$ and $M_H$
of sizes at least $\frac{n(G)-n(H)+1}{3}$ and $\frac{n(H)-1}{3}$, respectively.
Since the matching $M_H \cup M^\prime$ is uniquely restricted in $G$,
it follows that
\begin{align*}
  \nu_{ur}(G) \geq \frac{n(H)-1 + n(G) -n(H) +1}{3}
	      = \frac{n(G)}{3}, 
 \end{align*}
which is a contradiction.
Hence, we may assume that $v=x$.
Let $G^\prime = G-(V(H) \cup \{v\})$.
By Lemma \ref{l6}, $H$ contains a uniquely restricted matching $M_H$
of size $\frac{n(H)-1}{3}$ that does not cover $u$ and $w$.
Since $m_G(V(H)\cup \{v\}) \leq 1$, it follows, by Claim \ref{c1}, that
$G^\prime$ is not in $\mathcal{G}$, which implies that
$G^\prime$ contains a uniquely restricted matching $M^\prime$
of size at least $\frac{n(G)-n(H)-1}{3}$.
Since the matching $M_H \cup M^\prime \cup \{uv\}$ is uniquely restricted in $G$,
it follows that
\begin{align*}
  \nu_{ur}(G) \geq \frac{n(H)-1 + n(G) -n(H) - 1}{3} +1
	      > \frac{n(G)}{3}, 
 \end{align*}
which is a contradiction.
\end{proof}
\begin{claim} \label{c3}
 No induced subgraph $H$ of $G$ is isomorphic to a graph in $\mathcal{B} \setminus \mathcal{B}_c$.
\end{claim}
\begin{proof}
Suppose, for a contradiction, that
$H$ is an induced subgraph of $G$ that is isomorphic to a graph in 
$\mathcal{B} \setminus \mathcal{B}_c$.
By Claim \ref{c2}, $H$ is not isomorphic to $G_5$ or $G_8$
since both have exactly two vertices of degree $2$.
Let $X$ denote the set of vertices of degree $2$ in $H$,
and let $G^\prime = G-(V(H) \setminus X)$.
Note that $|X| = 3$, which implies that $G^\prime$ contains at most $3$ components.
Since $G$ is not isomorphic to $H$, there is a vertex $u \in X$ that is not isolated in $G^\prime$.
Let $H_1$ be the component of $G^\prime$ that contains $u$.
If $X \cap V(H_1) = \{u\}$, then, by Claim \ref{c2}, $H_1 \not\in \mathcal{G}$.
Therefore, at most two components of $G^\prime$ belong to $\mathcal{G}$.
If $X \cap V(H_1) \setminus \{u\} \neq \emptyset$, 
then $G^\prime$ contains at most $2$ components.
In both cases, $G^\prime$ contains at most $2$ components that belong to 
$\mathcal{G}$, which implies that $G^\prime$ contains a uniquely restricted matching 
$M^\prime$ of size at least $\frac{n(G)-n(H)+3-2}{3}$.
As illustrated in Figure \ref{fig2} and \ref{fig3}, $H$ contains 
a uniquely restricted matching $M_H$ of size $\frac{n(H)-1}{3}$
that does not cover the vertices in $X$
such that there is no $M_H$-alternating path between any
two vertices in $X$, which implies that the matching
$M^\prime \cup M_H$ is uniquely restricted in $G$.
This implies that
\begin{align*}
  \nu_{ur}(G) \geq \frac{n(G)-n(H)+1 + n(H)-1}{3}
	      = \frac{n(G)}{3}, 
 \end{align*}
which is a contradiction.
\end{proof}
\begin{claim} \label{c4}
 $G$ is not cubic.
\end{claim}
\begin{proof}
Suppose, for a contradiction, that $G$ is cubic.
Let $u$ be a leaf of some spanning tree of $G$,
and let $G^\prime = G-u$. Note that $G^\prime$ is connected.
If $G^\prime$ does not belong to $\mathcal{G}$,
then 
\begin{align*}
 \nu_{ur}(G) \geq \nu_{ur}(G^\prime) \geq \frac{n(G)-1}{3},
\end{align*}
which is a contradiction.
Hence, we may assume that $G^\prime$ belongs to $\mathcal{G}$.
Clearly, $G^\prime$ is not isomorphic to $G_2$ or $G_4$
since otherwise $G$ would be isomorphic to $H_1$ or $H_2$.
By Claim \ref{c3}, $G^\prime \in \mathcal{G}_c$ and,
by the girth condition, no vertex in $V(G^\prime) \setminus N_G(u)$
has $3$ neighbors in $N_G(u)$,
which, by Lemma \ref{l7}, implies that 
$G^\prime$ contains a uniquely restricted matching $M^\prime$
of size at least $\frac{n(G)-2}{3}$
that does not cover the vertices in $N_G(u)$.
Let $v$ be any neighbor of $u$.
The matching
$M^\prime \cup \{uv\}$ is uniquely restricted in $G$,
which implies that
\begin{align*}
 \nu_{ur}(G) \geq \frac{n(G)-2}{3}+1 >\frac{n(G)}{3},
\end{align*}
a contradiction.
\end{proof}

Claim \ref{c1} and \ref{c4} imply that $\delta(G) = 2$.
Let $u$ be a vertex of degree $2$, and
let $v$ and $w$ be the neighbors of $u$.
By the girth condition, $v$ and $w$ are not adjacent,
and the only common neighbor of $v$ and $w$ is $u$.
Let $G^\prime = G-\{u,v,w\}$.
\begin{claim} \label{c5}
 $G^\prime \in \mathcal{G}$ and $d_G(v) + d_G(w) \geq 5$.
\end{claim}
\begin{proof}
 Since $\delta(G) = 2$, $G^\prime$ contains 
 no isolated vertices, which, 
 by Claim \ref{c2},
 implies that at most one component of $G^\prime$
 belongs to $\mathcal{G}$.
 If no component of $G^\prime$
 belongs to $\mathcal{G}$,
 then $G^\prime$ contains a uniquely restricted matching $M^\prime$
 of size at least $\frac{n(G^\prime)}{3}$.
 Since the matching $M^\prime \cup \{uv\}$
 is uniquely restricted in $G$,
 we obtain that
 \begin{align*}
  \nu_{ur}(G) \geq \frac{n(G)-3}{3} + 1 = \frac{n(G)}{3},
 \end{align*}
 which is a contradiction.
 Hence, exactly one component of $G^\prime$ belongs to $\mathcal{G}$,
 which, by Claim \ref{c2}, implies that $d_G(v) + d_G(w) \geq 5$.
 Suppose that $G^\prime$ contains a second component $H$.
 Since $m_G(V(G^\prime)) \leq 4$, Claim \ref{c2} implies that $m_G(V(H)) = 1$.
 This implies, by Claim \ref{c1}, that
 neither $H$ nor $G-V(H)$ belong to $\mathcal{G}$.
 Therefore, $H$ and $G-V(H)$ contain 
 uniquely restricted matchings
 $M_H$ and $M^\prime$ of sizes at least
 $\frac{n(H)}{3}$ and
 $\frac{n(G)-n(H)}{3}$, respectively.
 Since the matching
 $M_H \cup M^\prime$ 
 is uniquely restricted in $G$,
 we obtain that
 \begin{align*}
 \nu_{ur}(G) \geq \frac{n(H) + n(G)-n(H)}{3} = \frac{n(G)}{3},
\end{align*}
which is a contradiction.
\end{proof}
Let $T$ be the host tree of $G^\prime$,
and let $H_1,\ldots,H_k$ be the blocks of $G^\prime$ isomorphic to a graph in $\mathcal{B}$,
which, by Claim \ref{c3}, implies that the blocks $H_1,\ldots,H_k$ belong to $\mathcal{B}_c$. 
Let $I \subseteq [k]$, let $G^{\prime\prime}$ arise from
$G^\prime$ by contracting the edges of $H_i$ for every $i$ in $I$,
and let $v_i$ be the vertex in $G^{\prime\prime}$ corresponding to $H_i$ for every $i$ in $I$.
Note that, if $I = [k]$, then $G^{\prime\prime}$ is isomorphic to $T$.
For the host tree $T$, let $D_T^\prime = D_T \setminus \{v_1,\ldots,v_k\}$.
\begin{claim} \label{c6}
 $d_G(v) + d_G(w) = 6$.
\end{claim}
\begin{proof}
Suppose, for a contradiction, 
that $v$ has degree $3$ and $w$ has degree $2$ in $G$.
Let $N_G(v) = \{u, x_1,x_2\}$, let $N_G(w) = \{u,x_3\}$,
and let $X = \{x_1,x_2,x_3\}$.
Since $G$ is not isomorphic to $G_5$ or $G_8$, it follows that
$G^\prime$ is not isomorphic to $G_2$ or $G_4$.
By the girth condition, 
no vertex in $V(G^\prime) \setminus X$ has $3$
neighbors in $X$.
Hence, Lemma \ref{l7} implies that
$G^\prime$ contains a uniquely restricted matching $M^\prime$
of size at least $\frac{n(G)-4}{3}$
that does not cover the vertices in $X$.
Since the matching $M^\prime \cup \{uv, wx_3\}$ is
uniquely restricted in $G$,
we obtain that
\begin{align*}
 \nu_{ur}(G) \geq \frac{n(G)-4}{3}+2 >\frac{n(G)}{3},
\end{align*}
which is a contradiction.
\end{proof}
Claim \ref{c6} implies that both $v$ and $w$ have degree $3$ in $G$.
Let $N_G(v) = \{u,x_1,x_2\}$, let $N_G(w) = \{u, x_3, x_4\}$, 
and let $X = \{x_1,x_2,x_3,x_4\}$.
\begin{claim} \label{c7}
$\nu_{ur}(G^\prime-X) < \nu_{ur}(G^\prime)$.
\end{claim}
\begin{proof}
If $G^\prime$ contains a uniquely restricted matching $M^\prime$
of size at least
$\nu_{ur}(G^\prime)=\frac{n(G^\prime)-1}{3}$
that does not cover the vertices of $X$,
then the matching $M^\prime \cup \{uv, wx_3\}$ is
uniquely restricted in $G$, which implies that
\begin{align*}
 \nu_{ur}(G) \geq \frac{n(G)-4}{3} + 2 > \frac{n(G)}{3},
\end{align*}
a contradiction.
\end{proof}
\begin{claim} \label{c8}
 There is some $i \in [k]$ such that $X \cap V(H_i) \neq \emptyset$.
\end{claim}
\begin{proof}
Suppose that $X \subseteq D_T^\prime$.
Lemma \ref{l2} $(iii)$ and Claim \ref{c7} imply that there is
a vertex $y_3 \in D_T^\prime$ with neighbors $y_1$ and $y_2$ in $A_T$
such that $X \subseteq N_{G^{\prime}}(\{y_1,y_2\})$.
By the girth condition and symmetry,
we may assume that $y_1$ is adjacent to $x_1$ and $x_3$,
and that $y_2$ is adjacent to $x_2$ and $x_4$, see Figure \ref{fig_c8}.
Since $D_T^\prime$ is independent, 
this implies that the graph induced by 
$\{u,v,w,y_1,y_2,y_3\} \cup X$ is isomorphic to $G_1$.
Let $T^\prime$ arise from $T$
by contracting all edges incident with $y_1$ or $y_2$,
and let $z$ be the newly created vertex.
$G$ arises from from $T^\prime$
by replacing the vertices 
$v_1,\ldots,v_k$
by $H_1,\ldots,H_k$ and $z$ by $G_1$, see Figure \ref{fig_c8}.
By Lemma \ref{l4}, $G \in \mathcal{G}$,
which is a contradiction.
\end{proof}
\begin{figure}[H] 
 \centering\tiny
 \begin{tikzpicture}[scale = 1.2]
\draw (4.5,0) ellipse(4 and 0.35);
\node[label=left:\footnotesize$y_1$] (y1) at (1,0) {};
\node[label=right:\footnotesize$y_2$] (y2) at (1.5,0) {};

\draw (2,1.2) ellipse (2 and 0.35);

\foreach \i in {1,2,3,4} {
\node[label=left:\footnotesize$x_{\i}$] (x\i) at (\i/2,1.2) {};
	    }
\node[label=right:\footnotesize$y_3$] (y3) at (2.5,1.2) {};

\node[label=above:\footnotesize$u$] (u) at (1.25,2.5) {};
\node[label=left:\footnotesize$v$] (v) at (0.75,2) {};
\node[label=right:\footnotesize$w$] (w) at (1.75,2) {};
	    
\draw (5.35,1.2) circle (0.35);
\draw (6.55,1.2) circle (0.35);
\draw (8.65,1.2) circle (0.35);

\draw[-,dotted] (7.2,1.2) -- (8,1.2);

\pgftext[x=3.25cm,y=1.2cm] {$D_T^\prime$};
\pgftext[x=4.5cm,y=0] {$A_T$};
\pgftext[x=5.35cm,y=1.2cm] {$H_1$};
\pgftext[x=6.55cm,y=1.2cm] {$H_2$};
\pgftext[x=8.65cm,y=1.2cm] {$H_k$};

\draw[-] (u) -- (v);
\draw[-] (u) -- (w);
\draw[-] (v) -- (x1);
\draw[-] (v) -- (x2);
\draw[-] (w) -- (x3);
\draw[-] (w) -- (x4);

\draw[-] (y1) -- (x1);
\draw[-] (y1) -- (x3);

\draw[-] (y2) -- (x2);
\draw[-] (y2) -- (x4);

\draw[-] (y1) -- (y3);
\draw[-] (y2) -- (y3);

\foreach \i [evaluate=\i as \x using 1.5+\i/2] in {1,2,3,4}
\draw[-,dotted] (x\i) -- (\x,0);

\draw[-,dotted] (y3) -- (4,0);

\end{tikzpicture}
 \caption{An illustration of Claim \ref{c8}.} \label{fig_c8}
 \end{figure}
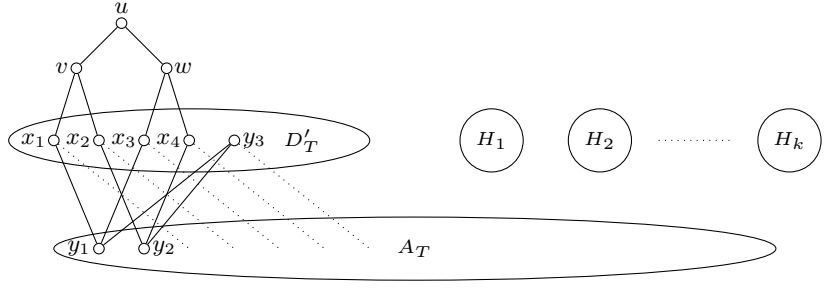 
\begin{claim} \label{c9}
 For all $i \in [k]$, we have $|X \cap V(H_i)| \neq 1$.
\end{claim}
\begin{proof}
Suppose, for a contradiction, 
that there is some $i \in \lbrack k \rbrack$
such that $|V(H_i) \cap X| = 1$.
By symmetry, we may assume that $x_1 \in V(H_i)$. 
Let $G^{\prime\prime}$ arise from $G^\prime$ by 
contracting the edges of $H_i$.
Since $n(T) \geq 2$ and $T$ is the host tree of $G^{\prime\prime}$,
it follows that
$G^{\prime\prime}$ is not isomorphic to $G_2$ or $G_4$.
Hence, Lemma \ref{l7}
implies that $G^{\prime\prime}$
contains a uniquely restricted matching $M$
of size $\frac{n(G^{\prime\prime})-1}{3}$
that does not cover the vertices in $X \setminus \{x_1\}$.
Let $M^\prime$ be the corresponding matching
in $G^\prime - X$.
Since $|V(M^\prime) \cap V(H_i)| \leq 1$, 
$H_i$ contains a uniquely restricted matching $M_i$
of size
$\frac{n(H_i)-1}{3}$
that does not cover $x_1$ and is disjoint from $M^\prime$, see Table \ref{table1}.
Since the matching
$M^\prime \cup M_i$
is uniquely restricted in $G^\prime - X$,
we obtain that 
\begin{align*}
 \nu_{ur}(G^\prime -X) \geq \frac{n(G^\prime) - n(H_i) + n(H_i)-1}{3} = \nu_{ur}(G^\prime),
\end{align*}
which is a contradiction to Claim \ref{c7}.
\end{proof}
\begin{claim} \label{c10}
 There is exactly one $i \in [k]$ with $X \cap V(H_i) \neq \emptyset$.
\end{claim}
\begin{proof}
Suppose, for a contradiction, that there are two indices $i\neq j$ such that 
$X \cap V(H_i) \neq \emptyset$ and $X \cap V(H_j) \neq \emptyset$.
Since $|X| = 4$, Claim \ref{c9} implies that
$|X \cap V(H_i)| = 2$ and
$|X \cap V(H_j)| = 2$.
By Lemma \ref{l2} $(i)$,
$T$ contains a matching $M$ of size 
$\frac{n(T)-1}{3}$
that does not cover $v_i$ and $v_j$.
Let $M^\prime$ be the corresponding
uniquely restricted matching in $G^\prime$.
Since $|V(M^\prime) \cap V(H_\ell)| \leq 1$ for $\ell \in [k]$
and $V(M^\prime) \cap V(H_\ell) = \emptyset$ for $\ell \in \{i,j\}$, 
it follows that
$H_1,\ldots,H_k$ contain 
uniquely restricted matchings $M_1,\ldots,M_k$ of 
sizes $\frac{n(H_1)-1}{3},\ldots,\frac{n(H_k)-1}{3}$
such that no vertex in $X$ is covered by
them and all are disjoint from $M^\prime$, see Table \ref{table1}.
Since the matching
$M^\prime \cup \bigcup_{i=1}^k{M_i}$
is uniquely restricted in $G^\prime - X$,
we obtain that
\begin{align*}
 \nu_{ur}(G^\prime -X) \geq \frac{n(T)-1}{3} + \sum_{i=1}^k{\frac{n(H_i)-1}{3}} = \nu_{ur}(G^\prime),
\end{align*}
which is a contradiction to Claim \ref{c7}.
\end{proof}
\begin{claim} \label{c11}
 There is some $i \in [k]$ with $|X \cap V(H_i)| \geq 3$.
\end{claim}
\begin{proof}
Suppose, for a contradiction, that there is some $i \in [k]$
such that $|X \cap V(H_i)| = 2$.
By Claim \ref{c10}, 
we may assume that $|X \cap D_T^\prime| = 2$.
By Lemma \ref{l2} $(ii)$ and Claim \ref{c7},
we may assume that the vertices in 
$X \cap D_T^\prime$ and $v_i$ 
have a common neighbor $y_1$ in $A_{T}$.
By the girth condition and symmetry,
we may assume that $X \cap D_T^\prime = \{x_2,x_4\}$
and $X \cap V(H_i) = \{x_1,x_3\}$.
Let $M$ be some maximum matching in $T$ that does
not cover the vertices in
$X \cap D_T^\prime$, 
and let $M^\prime$ be the corresponding
uniquely restricted matching in $G^\prime$.
Since $v_i$ is covered by $M$,
it follows that 
$V(H_i) \cap V(M^\prime) \neq \emptyset$, say
$y_2 \in V(H_i) \cap V(M^\prime)$.
Note that, by the construction of $\mathcal{G}$,
$y_2 \not\in V(H_i) \cap X$.
By Claim \ref{c7}, 
we may assume that $H_i$ does not contain a
maximum uniquely restricted matching
that does not cover $y_2$ and the vertices
in $X \cap V(H_i)$.
This implies that
$H_i$ is not isomorphic to $G_3$, see Table \ref{table1}.
If $H_i$ is isomorphic to $G_1$, then 
the vertices in 
$X \cap V(H_i)$ and $y_2$ 
have a common neighbor $y_3$, see Table \ref{table1}.
Hence, the subgraph induced by $V(H_i) \cup \{u,v,w,x_2,x_4,y_1\}$
is isomorphic to $G_3$, see Figure \ref{fig_c11}.
Let $T^\prime$ arise from $T$
by contracting all edges incident with $y_1$,
and let $z$ be the newly created vertex.
$G$ arises from $T$ by replacing
the vertices $v_1,\ldots,v_{i-1}, v_{i+1},\ldots,v_k$ 
by $H_1,\ldots,H_{i-1}, H_{i+1},\ldots,H_k$ 
and $z$ by $G_3$,
which, by Lemma \ref{l3}, 
implies that $G \in \mathcal{G}$, a contradiction.
Hence, $H_i$ is isomorphic to $G_2$ or $G_4$,
which implies that the subgraph induced by $V(H_i) \cup \{u,v,w,x_2,x_4,y_1\}$
is isomorphic to $G_6$ or $G_9$, see Figure \ref{fig_c11}, a contradiction to Claim \ref{c3}.
\end{proof}
\begin{figure}[H] 
 \centering\tiny
 \begin{tikzpicture}[scale = 1.2]
\draw (4.5,0) ellipse(4 and 0.35);
\node[label=left:\footnotesize$y_1$] (y1) at (4.5-2.1,0) {};
\node[label=right:\footnotesize$y_2$] (y2) at (4,0.95) {};

\draw (1.5,1.2) ellipse (1.5 and 0.35);

\foreach \i in {2,4} {
\node[label=left:\footnotesize$x_{\i}$] (x\i) at (\i/2+0.75,1.2) {};
	    }

\node[label=below:\footnotesize$x_1$] (x1) at (3.75,1.45) {};
\node[label=below:\footnotesize$x_3$] (x3) at (4.25,1.45) {};

\node[label=above:\footnotesize$u$] (u) at (2.25,2.5) {};
\node[label=left:\footnotesize$v$] (v) at (1.75,2) {};
\node[label=left:\footnotesize$w$] (w) at (2.75,2) {};
	    
\draw (4,1.2) circle (0.5);
\draw (6.55,1.2) circle (0.35);
\draw (8.65,1.2) circle (0.35);

\draw[-,dotted] (7.2,1.2) -- (8,1.2);

\pgftext[x=0.5cm,y=1.2cm] {$D_T^\prime$};
\pgftext[x=4.5cm,y=0] {$A_T$};
\pgftext[x=5cm,y=2cm] {$H_i$};
\draw[-latex] (5,1.9) to[out=-90,in=20] (4.6,1.2);
\pgftext[x=6.55cm,y=1.2cm] {$H_1$};
\pgftext[x=8.65cm,y=1.2cm] {$H_k$};

\draw[-] (u) -- (v);
\draw[-] (u) -- (w);
\draw[-] (v) -- (x1);
\draw[-] (v) -- (x2);
\draw[-] (w) -- (x3);
\draw[-] (w) -- (x4);

\draw[-] (y1) -- (x2);
\draw[-] (y1) -- (x4);
\draw[-] (y1) -- (y2);

\end{tikzpicture}
 \caption{An illustration of Claim \ref{c11}.} \label{fig_c11}
 \end{figure}
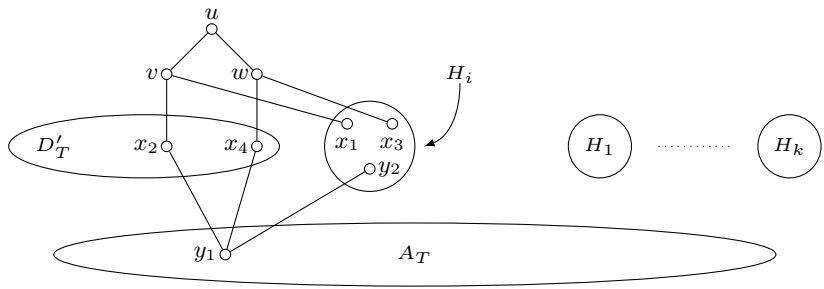 
\begin{claim} \label{c12}
 There is some $i \in [k]$ such that $X \subseteq V(H_i)$.
\end{claim}
\begin{proof}
Suppose, for a contradiction, that $X \cap D_T^\prime \neq \emptyset$.
Claim \ref{c11} implies that $|X \cap V(H_i)| = 3$ 
for some $i \in \lbrack k \rbrack$.
Therefore, $|X \cap D_T^\prime| = 1$.
By symmetry, we may assume that $x_4 \in X \cap D_T^\prime$.
Since there is some vertex $y$ of degree $2$ in $H_i$
that is adjacent to some vertex in $A_T$,
it follows that $H_i$ is not isomorphic to $G_2$ or $G_4$
since $y \not\in X$. Therefore, 
we may assume that $H_i$ is isomorphic to $G_1$ or $G_3$.
By Lemma \ref{l6}, the graph $G^{\prime\prime}$ that arises
from $G^\prime$ by contracting the edges
of $H_i$ contains a uniquely restricted matching $M^\prime$
of size $\frac{n(G^\prime)-n(H_i)}{3}$
that does not cover $v_i$ and $x_4$.
Furthermore, by the girth condition, 
the vertices in $X \cap V(H_i)$
do not have a common neighbor in $G^\prime$,
which implies that
$H_i$ contains a uniquely restricted matching $M_i$
of size $\frac{n(H_i)-1}{3}$ that does not 
cover the vertices in $X \cap V(H_i)$, see Table \ref{table1}.
Since the matching
$M^\prime \cup M_i$ is uniquely restricted in 
$G^\prime-X$, we obtain that 
\begin{align*}
 \nu_{ur}(G^\prime -X) \geq \frac{n(G^\prime) - n(H_i) + n(H_i) - 1}{3} = \nu_{ur}(G^\prime),
\end{align*}
which is a contradiction to Claim \ref{c7}.
\end{proof}
We are now in a position to complete the proof.
By Claim \ref{c12}, we may assume that $|X \cap V(H_i)| = 4$ 
for some $i \in \lbrack k \rbrack$.
Therefore, $H_i$ must be isomorphic
to $G_1$ or $G_3$.
First, we assume that $H_i$ is isomorphic to $G_1$.
Using the names specified in the left of Figure \ref{fig_cfinal},
we may, by symmetry, assume that $x_1=v_1$,
which, by the girth condition,
implies that $x_2 = v_6$.
Hence, $w$ is either adjacent to $v_3$ and $v_5$ 
or to $v_2$ and $v_4$. In both cases,
the subgraph induced by $V(H_i) \cup \{u,v,w\}$
is isomorphic to $G_2$, which implies that $G \in \mathcal{G}$,
a contradiction.

Hence, we may assume that $H_i$
is isomorphic to $G_3$.
Using the names specified in the right of Figure \ref{fig_cfinal},
we may, by symmetry, assume that $x_1=v_6$.
By symmetry between $v_2$ and $v_4$,
we may, by the girth condition, assume that $x_3 = v_4$.
By Table \ref{table1} and Claim \ref{c7}, we may assume that $H_i$ contains a vertex
$y_3$ with neighbors $y_1$ and $y_2$ such that 
$X \subseteq N_{H_i}(\{y_1,y_2\})$.
By symmetry and the girth condition, we may assume that
$y_1$ has neighbors
$x_1$ and $x_3$ and that $y_2$ has neighbors $x_2$ and $x_4$.
Therefore, $y_1 = v_5$, which implies that $y_3 = w_3$.
Since $x_2$ and $x_4$ have degree $2$ in $H_i$,
it follows that $y_2 \neq z$, that is, $y_2 = u_5$.
Thus, $\{x_2,x_4\} = \{u_4,u_6\}$.
If $x_2 = u_6$, then $x_4 = u_4$, and 
the subgraph induced by $V(H_i) \cup \{u,v,w\}$ is isomorphic to $G_4$,
which implies that $G \in \mathcal{G}$.
Hence, we may assume that $x_2 = u_4$, which implies that $x_4 = u_6$.
Therefore, the subgraph induced by $V(H_i) \cup \{u,v,w\}$
is isomorphic to $G_7$, which is a contradiction to Claim \ref{c3}.
\end{proof}

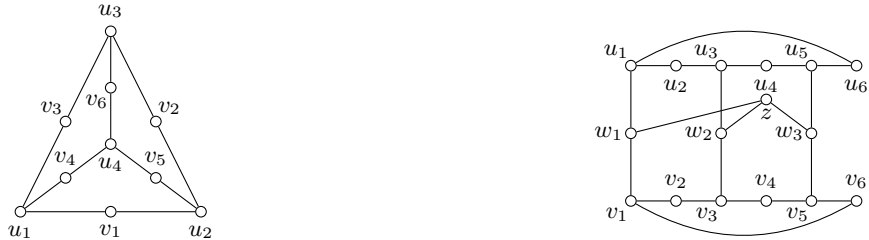
\begin{figure}[H] 
 \begin{minipage}[b]{0.49\textwidth}
 \centering\tiny
\begin{tikzpicture}[scale = 1.2]
	    \node[label=below:\footnotesize$u_1$]  (u1) at (0,0) {};
	    \node[label=below:\footnotesize$u_2$]   (u2) at (2,0) {};
	    \node[label=above:\footnotesize$u_3$]   (u3) at (1,2) {};
	    \node[label=below:\footnotesize$u_4$]   (u4) at (1,0.75) {};
	    \node[label=below:\footnotesize$v_1$]   (v1) at (1,0) {};
	    \node[label=above right:\footnotesize$v_2$]   (v2) at (1.5,1) {};
	    \node[label=above left:\footnotesize$v_3$]   (v3) at (0.5,1) {};
	    \node[label=above:\footnotesize$v_4$]   (v4) at (0.5,0.375) {};
	    \node[label=above:\footnotesize$v_5$]   (v5) at (1.5,0.375) {};
	    \node[label=below left:\footnotesize$v_6$]   (v6) at (1,1.375) {};
	    
	    \foreach \from/\to in {u1/v1, v1/u2, v2/u2, v2/u3, v3/u1, v3/u3, v4/u1, v4/u4, v5/u2, v5/u4, v6/u3, v6/u4}
	    \draw [-] (\from) -- (\to);
\end{tikzpicture}
 \end{minipage}
\begin{minipage}[b]{0.49\textwidth}
 \centering\tiny
\begin{tikzpicture}[scale = 1.2]

	    \foreach \i in {2,4,6} {
	    \node[label=above:\footnotesize$v_{\i}$]  (v\i) at (-0.5+\i/2,0.25) {};	    
	    }
	    \foreach \i in {1,3,5} {
	    \node[label=below left:\footnotesize$v_{\i}$]  (v\i) at (-0.5+\i/2,0.25) {};	    
	    }
	    
	    \foreach \i in {2,4,6} {
	    \node[label=below:\footnotesize$u_{\i}$]  (u\i) at (-0.5+\i/2,1.75) {};	    
	    }

	     \foreach \i in {1,3,5} {
	    \node[label=above left:\footnotesize$u_{\i}$]  (u\i) at (-0.5+\i/2,1.75) {};	    
	    }
	    
	    \node[label=left:\footnotesize$w_{1}$] (w0) at (0,1) {};
	    \node[label=left:\footnotesize$w_{2}$] (w1) at (1,1) {};
	    \node[label=left:\footnotesize$w_{3}$] (w2) at (2,1) {};

	    \foreach \i/\j in {0/1,1/3,2/5} {
	    \draw[-] (w\i) -- (v\j);
	    \draw[-] (w\i) -- (u\j);
	    }

	    \foreach \i/\j in {1/2, 2/3, 3/4, 4/5, 5/6} {
	    \draw[-] (v\i) -- (v\j);
	    \draw[-] (u\i) -- (u\j);
	    }
	    
	    \node[label=below:\footnotesize$z$] (z) at (1.5,1.375) {};
	    \foreach \i in {0,1,2} 
	    \draw[-] (z) -- (w\i);
	    
	    \draw[-] (v1) to[out=-30,in=-150] (v6);
	    \draw[-] (u1) to[out=30,in=150] (u6);
	    
\end{tikzpicture}
\end{minipage}
\caption{An illustration of the final contradiction.} \label{fig_cfinal}
\end{figure}

\end{document}